\documentclass[a4paper,12pt, reqno, oneside]{article}
\usepackage{amsmath, amssymb, amsthm,graphicx,xcolor}
\usepackage{xspace,mathrsfs,bbm,enumitem}
\usepackage[numbers]{natbib}
\usepackage[T1]{fontenc}
\usepackage[scaled=0.93]{newtxtext,newtxmath}
\usepackage{enumitem}
\newcommand{\bbf}{\mathbbm}
\setlist{noitemsep}
\setlist{nosep}

\definecolor{link}{rgb}{0,0,0}
\definecolor{cit}{rgb}{0,0,0.6}
\definecolor{cite}{named}{cit}

\usepackage[
pdftitle={},
pdfauthor={},
bookmarks=false,
colorlinks=true,
urlcolor=url,
linkcolor=link,
citecolor=cite,
filecolor=black,
raiselinks=false,
hyperfootnotes=true]{hyperref}

\usepackage{titlesec}

\titleformat{\section}{\center\large\bf}{\thesection.}{.7ex}{}
\titlespacing{\section}{0pc}{4ex plus .5ex minus .5ex}{2ex minus .1ex}

\titleformat{\subsection}{\center\bf}{\thesubsection.}{.7ex}{}
\titlespacing{\subsection}{0pc}{2ex plus .1ex minus .2ex}{2ex minus .1ex}

\newtheorem{theorem}{Theorem}
\newtheorem{lemma}[theorem]{Lemma}

\theoremstyle{definition}

\newtheorem{remark}{Remark}
\newtheorem{example}{Example}

\newcommand{\op}[1]{{\varmathbb #1}}

\newcommand{\PP}{{\op P}}
\newcommand{\EE}{{\op E}}
\newcommand{\pp}[1]{{\op P}(#1)}
\newcommand{\ee}[1]{{\op E}(#1)}

\newcommand{\e}{{\rm e}}

\newcommand{\la}{\lambda}

\newcommand{\LN}[1]{N_{#1}}
\newcommand{\LD}[1]{D_{#1}}

\newcommand{\LL}[1]{\Phi_{#1}}

\newcommand{\sq}[1]{\op #1}

\newcommand{\sqNN}{\sq N_0}
\newcommand{\seq}[1]{\{#1_i\}_{i\in\sqNN}}

\newcommand{\deq}{{=}_{\rm d}}
\newcommand{\dto}{\Rightarrow}

\newcommand{\VV}{V}
\newcommand{\QQ}[1]{\op Q[#1]}
\newcommand{\re}[1]{\text{\rm Re}\,#1}

\newcommand{\veps}{\varepsilon}

\newcommand{\MW}{M}
\newcommand{\aplus}[1]{\left[#1\right]^+}
\newcommand{\aminus}[1]{\left[#1\right]^-}

\newcommand{\fexp}[1]{{\rm exp}\left(#1\right)}
\newcommand{\afrac}[2]{#1/#2}
\newcommand{\WA}{W^\ast}
\newcommand{\gen}[1]{U_{#1}}

\newcommand{\abs}[1]{|#1|}

\newcommand{\VM}{V^-}

\newcommand{\ind}[1]{{\bbf 1}_{\{#1\}}}
\newcommand{\kla}[1]{\left(#1\right)}
\newcommand{\bkla}[1]{\big(#1\big)}
\newcommand{\Bkla}[1]{\Big(#1\Big)}

\newcommand{\al}{\alpha}
\newcommand{\alf}{^{(\alpha)}}

\begin{document}

\hfill{\small\itshape Draft, \today}
\vspace{4ex}
\begin{center}
\parbox{0.87\textwidth}{\huge
\center
\Large\bf A Multiplicative Version of the Lindley Recursion}\\[4ex]
{\sc Onno Boxma\\Andreas L\"opker\\Michel Mandjes\\Zbigniew  Palmowski}
\end{center}
\vspace{2ex}

\begin{abstract} \noindent This paper presents an analysis of the  stochastic recursion $W_{i+1} = [V_iW_i+Y_i]^+$ that can be interpreted as an autoregressive process of order 1, reflected at 0. We start our exposition by a discussion of the model's stability condition. Writing $Y_i=B_i-A_i$, for independent sequences of non-negative i.i.d.\ random variables $\seq A$ and $\seq B$, and assuming $\seq V$ is an i.i.d.\ sequence as well (independent of $\seq A$ and $\seq B$), we then consider three special cases: (i)~$V_i$ attains negative values only
and $B_i$ has a rational LST, (ii) $V_i$ equals a positive value $a$ with certain probability
$p\in (0,1)$ and is negative otherwise, and  both $A_i$ and $B_i$ have a rational LST,
(iii) $V_i$ is uniformly distributed on $[0,1]$, and $A_i$ is exponentially distributed.
In all three cases we derive transient and stationary results,
where the transient results are in terms of the transform at a geometrically distributed epoch.
\\

\noindent
{\bf AMS Subject Classification (MSC2010).} Primary: 60K25;
Secondary: 90B22.
\\

\noindent
{\bf Affiliations.} O. Boxma is with Eurandom and the Department of Mathematics and Computer Science; Eindhoven University of Technology; P.O. Box 513, 5600 MB Eindhoven; The Netherlands ({\it email}: {\tt  \footnotesize o.j.boxma@tue.nl}).
A. L\"opker is with HTW Dresden,
University of Applied Sciences,
%Hochschule für Technik und Wirtschaft
Friedrich-List-Platz 1,
D-01069 Dresden; Germany ({\it email}: {\tt \footnotesize lopker@htw-dresden.de}).
M. Mandjes is with Korteweg-de Vries Institute for Mathematics, University of Amsterdam; Science Park 904, 1098 XH Amsterdam; The Netherlands ({\it email}: {\tt \footnotesize m.r.h.mandjes@uva.nl}).
Z. Palmowski is with Department of Applied Mathematics, Faculty of Pure and Applied Mathematics,
Wroc\l aw University of Science and Technology, Wyb. Wyspia\'nskiego 27, 50-370 Wroc\l aw, Poland
({\it email}: {\tt \footnotesize zbigniew.palmowski@pwr.edu.pl}).
\\

\noindent
{\bf Acknowledgments.}   The research of Boxma and Mandjes is partly funded by the NWO Gravitation Programme NETWORKS (Grant Number 024.002.003) and an NWO Top Grant (Grant Number 613.001.352). The research of Palmowski is partially supported by Polish National Science Centre Grant No. 2018/29/B/ST1/00756, 2019-2022.

\end{abstract}

\newpage

\section{Introduction}
This paper focuses on the Lindley type stochastic recursion
\begin{equation}
W_{i+1} = [V_iW_i+Y_i]^+, ~~~i=0,1,\dots,
\label{mainsr}
\end{equation}
where $\aplus x=\max\{x,0\}$ for $x\in\sq R$
and $\seq \VV$ and $\seq Y$ are independent sequences of i.i.d. (independent, identically distributed) random variables.
The analysis of stochastic recursions has received much attention in the applied probability literature.
This holds in particular for stochastic recursions of the autoregressive type, owing to their wide applicability
across various scientific domains including biology, finance, and engineering \cite{BMR, AB, embr}.

An important subclass of first order autoregressive models corresponds to the case in which the $\seq V$ are constant, i.e., a stochastic process defined through the recursion
\begin{equation}\label{Wi}
W_{i+1} = aW_i +Y_i, ~~~i=0,1,\dots,
\end{equation}
for a sequence of i.i.d.\ random variables $\seq Y$ and a scalar $a$, with $W_0$ being given.
When the quantities $W_i$ cannot attain negative values, then it becomes natural to study the truncated counterpart of \eqref{Wi}, i.e.,
 the recursion
\begin{equation}
W_{i+1} = [aW_i +Y_i]^+, ~~~i=0,1,\dots .
\label{mainspec}
\end{equation}
When $a=1$ we recover the classical Lindley recursion describing the waiting time in the G/G/1 queue, with $Y_i$ representing the difference between the $i$-th
service time and the ($i+1$)-st interarrival time.
The case of $a \in (0,1)$ was studied  in detail in  \cite{BMR}, whereas the case $a=-1$ is covered by  \cite{VlasiouPhD}. It should be observed that, while from the analysis it is clear that $a$ is assumed to be positive in \cite{BMR},
the introduction of that paper incorrectly states that $|a| < 1$.

By studying (\ref{mainsr}), we significantly extend the analysis of  the Lindley recursion as well as the analysis of the stochastic recursion (\ref{mainspec}).
Our results focus on three different choices of $\{V_i\}$.
In Model I, the $V_i$  are negative random variables. A detailed analysis is shown to be possible as long as the positive part of the $Y_i$
has a rational Laplace-Stieltjes transform (in the sequel abbreviated to LST).
In Model II, the $V_i$ are either negative or equal to the positive constant $a$.
Here we demand that both the positive and the negative part of the $Y_i$ have a rational LST.
While Model II effectively contains Model I, we prefer to give a separate analysis of both models, to make the reader familiar
with the specific mathematical intricacies due to $V_i$ being negative (Model I) and $V_i$ being a positive constant (the model in \cite{BMR}).
Finally, in Model III, the $V_i$ are uniformly distributed on $[0,1]$, and the negative part of the $Y_i$ is exponentially distributed; this case requires an entirely different approach.

The rationality assumptions are natural in the light of the existing theory that has been developed for the G/G/1 queue.
While in principle the waiting-time distribution in the general G/G/1 queue can be obtained via a Wiener-Hopf decomposition (cf.\ \cite[Chapter II.5]{cohen}),
the solution is a rather implicit one, unless one makes rationality assumptions on either the interarrival or the service-time LST. In addition, it can be argued that the distribution of any non-negative random variable can be approximated arbitrarily closely by the distribution of a random variable with a rational LST \cite[Ch.\ III]{asm}, so that a restriction to random variables with rational LST leads to just a minor loss of generality.

Notable studies of  stochastic recursions are \citet{Foss1, seva}, and \citet{diaconis}; see in addition \cite{bor1}.
For the non-reflected case, stochastic recursions of the form $W_{i+1} = \VV_iW_i + Y_i$ have  been studied frequently, partly under the name
`Vervaat perpetuity'; we mention \cite{AB,devroye, embr, horst, kesten, vervaat}.
%When $\VV=1$ and $Y_i=B_i-A_i$ as above then \eqref{recur1} is known as Lindley's equation \cite{lindley, kleinrock}.
%In \cite{VlasiouPhD} the case of a Lindley type relation with $\VV=-1$ has been investigated.
%Processes with reflection, related to the one discussed here, have been investigated in \citet{BMR}, where $V_i=a$ is a constant
%(from the analysis in \cite{BMR} it is clear that $a>0$, but the introduction of \cite{BMR} incorrectly suggests that $|a|<1$).
For the reflected case,
\cite{WOM} considers another generalization of the Lindley recursion, by replacing $V_iW_i$ in (\ref{mainsr}) by $S(W_i)$, where $\{S(t)\}_{t\geqslant 0}$ is a L\'evy subordinator. A model that is similar to the present model has been discussed in \citet{whitt}. It is noted, though, that \cite{whitt} primarily focuses on stability questions,  limit theorems and questions related to  queuing applications, whereas our primary focus lies on the derivation of results for the  transient and stationary distribution of the process under investigation.
Also related is the model in \cite{BoxmaVlasiou}; there (\ref{mainsr}) is considered with $\PP(V_i=1)=p$, $\PP(V_i=-1)=1-p$.

The main contributions of the present paper are the following.
For each of the three models that we introduced above, we state and solve a Wiener-Hopf boundary value problem,
which allows us to study the transient behavior of the $\seq W$ process.
In particular, we obtain an expression for the object
\[\sum_{i=0}^{\infty} r^i \,\ee{{\rm e}^{-sW_i}},\]
which can be interpreted as the generating function of the LST of
the $W_i$, but also (up to the multiplicative constant $1-r$) as the LST after a geometrically distributed time.
The stability condition of each of the three models is discussed, and the steady-state distribution of the $\seq W$ process is also determined.

The remainder of the paper is organized as follows.
Section~\ref{modeldes} presents the description of the three models and some preliminaries.
Sections~\ref{Case1}, \ref{Case2} and \ref{Case3} are devoted to the transient and steady-state analysis of, respectively,  Models I, II and III.
Section~\ref{concl} contains a discussion and provides  suggestions for further research.

\section{Model description and preliminaries}
\label{modeldes}
The main object of study is  the stochastic recursion
\begin{align}
W_{i+1} = \aplus{\VV_iW_i + Y_i}, \quad i=0,1,\dots,
\label{recur1}
\end{align}
where $\seq \VV$ and $\seq Y$ are sequences of  i.i.d.\  random variables, which are in addition independent of each other. The initial state of the process is assumed to be $W_0=w\in\sq R^+$. We write $\VV$ and $Y$ for generic random variables distributed as $\VV_0$ and $Y_0$ respectively.

In this paper, we discuss the following three variants of the model:
\begin{center}
\begin{tabular}{ll}
Model I: & $\PP(V<0)=1;$\\
Model II:& $\PP(V=a)=p,\ \PP(V<0)=1-p,\ a>0,\ p\in(0,1)$;\\
Model III:& $\PP(V<x) = x, \ 0 \leqslant x \leqslant 1$.
\end{tabular}
\end{center}
In each of the cases we will assume that the $Y_i$ are decomposed as $B_i-A_i$, with sequences $\seq{A}$ and $\seq{B}$ of i.i.d.\ {\it non-negative} random variables. In addition, depending on the chosen model, the random variables $A_i$ and/or $B_i$ are assumed to have a rational LST.

We start with investigating the stationary behavior of $\seq W$. We always assume that both $\EE|V|$ and   $\EE|Y|$ are finite. The first result was given in \citet{whitt} and covers most cases of interest.

\begin{theorem}[\citet{whitt}]
\label{Theorem1}
If one of the following conditions holds, then $W_i$ tends weakly to  a proper limit $W$ as $i\to\infty$:
\begin{enumerate}[label={\rm(C$_\arabic*$)}]
\item $\pp{V<0}>0$ and $\pp{Y\leqslant 0}>0$, \label{w1}
\item $V\geqslant 0$ a.s. and $\pp{V=0}>0$, \label{w2}
\item $V>0$ a.s. and $\ee{\log|V|}<0$.\label{w3}
\end{enumerate}
Moreover, $W_i$ converges weakly to a possibly improper limit $W$ as $i\to\infty$ if
$V>0$ a.s., $\ee{\log|V|}=0$, and $W_0=0$. If additionally $V=1$ a.s. then $W$ is proper for $\ee{Y}<0$ and improper for $\ee{Y}>0$.
\end{theorem}
It follows straightforwardly from the regenerative structure of $W_i$, and the proof of the above theorem in
\citet{whitt}, that in cases \ref{w1} and \ref{w2} the limit $W$ is unique.
Obviously under the conditions of the theorem the limiting  random variable $W$ fulfils the associated distributional identity $W\deq \aplus{\VV W + Y}$.
Regarding the above Condition \ref{w3}, we add the following observation.

\begin{theorem}\label{theo2} In order to have convergence of $W_i$ to a proper unique limit $W$ as $i\to\infty$, it is sufficient to have
 $\ee{\log|V|}<0$, which in turn is implied by $\EE\abs{V}<1$.
\end{theorem}

\begin{proof} Recursion \eqref{recur1} can be written as a random iteration $W_{i+1}=f_{\theta_i}(W_i)$ with $f_{\theta_i}(x)=\aplus{xV_i+Y_i}$, $\theta_i=(V_i,Y_i)$. This means that $f_\theta(\cdot)$ enjoys the Lipschitz property
\begin{align*}
|f_{\theta}(x)-f_{\theta}(y)|\leqslant K_\theta|x-y|,
\end{align*}
with random Lipschitz constant $K_\theta=|V|$. As a result, \citet[Thm.\ 1.1]{diaconis} is applicable. By Jensen's inequality $\EE \log |V|\leqslant \log\EE|V|$ and so $\EE|V|<1$ implies condition $\ee{\log|V|}<0$.
\end{proof}

The case where $\pp{V<0}>0$ and $Y\geqslant 0$ a.s., which was omitted in \citet{whitt}, is more involved due to the fact that the process might not be aperiodic, even if $Y$ is not deterministic. As an example suppose that the distribution of $Y$ is supported on $[1,2]$ and that $V\leqslant -2$ a.s. If $W_0=0$, then $W_1\in[1,2]$, $W_2=0$, $W_3\in[1,2]$, entailing that the process alternates between the set $\{0\}$ and a value in $[1,2]$. On the other hand, if $W_0>2$, then $W_1=0$, $W_2\in[1,2]$ and so on. As a consequence, there is no convergence $W_i\dto W$ as $i\to\infty$. However, regarding the existence of a stationary distribution we can show the following.

\begin{theorem}
If $\pp{V\leqslant 0}>0$ and $Y\geqslant 0$ a.s., then there is convergence of $W_i$ to a  stationary random variable $W$  as $i\to\infty$.
\end{theorem}
\begin{proof}

We define a  majorizing process by $\MW_0:=W_0$ and
$\MW_{i+1}:=\aplus{V_i}\MW_i+Y_i$. Then $W_i\leqslant \MW_i$, $i=0,1,2,\ldots$ and for $\seq{\MW}$ we have $\ee{\log|\aplus{V}|}=-\infty<0$ since $\pp{\aplus{V}=0}>0$. So by Thm.~\ref{theo2} it follows that $\MW_i\dto \MW$ as $i\to\infty$ for some limiting random variable $\MW$. This implies the existence of a stationary distribution for $\seq{W}$. Indeed, let $\veps>0$ and let $k>0$ be such that $\pp{\MW>k}\leqslant \veps/2$. Letting $b_i=\pp{W_i\leqslant k}$, we obtain
$b_i\geqslant \pp{\MW_i\leqslant k}$. The right-hand side converges to $\pp{\MW\leqslant k}$ as $i\to\infty$
which is at least $1-\veps/2$, implying that there is an $i_0$ such that for $i\geqslant i_0$, $b_i>1-\veps$ for all $i\geqslant i_0$. Consequently the family $\pp{W_i\leqslant \cdot}$ is tight, guaranteeing \cite[Thm.\ 4]{sergeytakis} the existence of a stationary distribution for $\seq{W}$.
\end{proof}

We end this section with a lemma that forms the starting-point of the analysis of all three models.
For this, we need to introduce some additional notation.
For a given non-negative random variable $X$ we write $\LL{X}(s)=\EE e^{-sX}$ for its  LST, defined at least for $\re s\geqslant 0$.
We say that $\LL{X}\in\QQ{s_1,s_2,\dots,s_n}$ if $X$ has a rational LST with poles at $s_1,s_2,\dots,s_n$, i.e., if $\LL X(s)$ is of the form
\begin{align}
\LL{X}(s)=\frac{\LN X(s)}{\LD X(s)},\label{rlt}
\end{align}
where $\LD X(s)=\prod_{i=1}^n (s-s_i)$ and
$\LN X(s)$ is a polynomial of degree at most $n-1$ not sharing zeros with $\LD X(s)$. Note that this implies that $\PP(X=0)=\lim_{s\to\infty}\LL X(s)=0$. With this notation we then have e.g.  $\LL{Y}(s)=\LL{B}(s)\,\LL{A}(-s)$. We also write $\LD Y(s)=\LD{B}(s)\,\LD{A}(-s)$ and $\LN Y(s)=\LN{B}(s)\,\LN{A}(-s)$ if $A$ and $B$ have rational LSTs of the form \eqref{rlt}.

For a sequence $\seq X$ of random variables we introduce the generating function, for $r\in(0,1)$:
\begin{align*}
\gen X(r,s)=\sum_{i=0}^\infty r^i\LL{X_i}(s).
\end{align*}
Note that since $\LL{\VV X}(s)=\int\LL{X}(ys)\,\PP(\VV\in {\rm d}y)$,
%\footnote{\scriptsize \tt MM: I am not sure $\PP(V\leqslant {\rm d}y)$ is adequate notation. It should be either $\PP(V\in{\rm d}y)$ (which I prefer) or ${\rm d}\PP(V\leqslant y)$. Informally, just to be sure: we mean the {\em density}, right?}
we have
\begin{align}
\gen {\VV X}(r,s)=\int \gen {X}(r,ys)\,\PP(\VV\in\,{\rm d}y).\label{note}
\end{align}
The following lemma plays a key role in our analysis. Define $\WA_i:=\aminus{\VV_iW_i+Y_i}$, where
$\aminus x:=\min\{x,0\}$.
\begin{lemma}\label{bas}  $\gen W(r,s)$ and $\gen{\WA}(r,s)$ are, for  $r\in(0,1)$ and $\re s = 0$, related via
\begin{align}
\gen W(r,s)  = \e^{-sw}+ r\Big(\LL Y(s) \gen{\VV W}(r,s) + \frac{1}{1-r} - \gen\WA(r,s)\Big)  .
\label{BASIC}
\end{align}
\end{lemma}
\begin{proof}
First observe that the basic identity $\fexp{\aplus{x}} = \fexp{x}+1-\fexp{\aminus x}$ applies.
It thus follows from \eqref{recur1} that
\begin{align}
\LL{W_{i+1}}(s) =\LL{\VV_iW_i+Y_i}(s) +1-\LL{\WA_i}(s), \quad i=0,1,\dots.
\label{recurc}
\end{align}
Multiplying both sides of \eqref{recurc} by $r^{i+1}$ and summing yields the identity \eqref{BASIC}.
\end{proof}

\section{Model I: The negative case}
\label{Case1}

%\footnote{\scriptsize \tt AL: Think about using $Y$ or $B-A$; if only $Y$ in the beginning, then adapt in the beginning of later sections.}
The model we analyze in this section assumes that each $V_i$ attains only negative values and that $Y_i$  is the difference $B_i - A_i$ of two independent non-negative random variables, where $B_i$ has a rational LST. In other words, we impose the conditions
\begin{itemize}
\item[(A)] $\VV<0$ a.s.,
\item[(B)] $\LL B\in \QQ{s_1,\dots,s_\ell}$ with $\re s_j<0$ for $j=1,\dots,\ell$.
\end{itemize}

\begin{theorem} Suppose that the Conditions {\em (A)} and {\em (B)} hold. Then, for $r\in(0,1)$,
\begin{align}
&\gen W(r,s) = \e^{-sw} + \frac{\sum_{k=0}^\ell a_k(r) s^k}{\LD B(s)}, ~~~
\re s\geqslant 0,
\label{W1}
\\
&\gen\WA(r,s)=\frac{1}{1-r} -\frac{\sum_{k=0}^\ell a_k(r) s^k}{r\LD B(s)}+\LL Y(s) \int_{-\infty}^0 \gen W(r,sy) \,\PP(\VV\in\,{\rm d}y) ,
~~~ \re s\leqslant 0,
\label{W2}
\end{align}
where
\begin{equation}\label{recur8}
a_0(r)= \frac{r}{1-r}(-1)^\ell \prod_{j=1}^\ell s_j,
\end{equation}
and the remaining constants $a_1(r),\dots,a_\ell(r)$ can be determined from the linear system \eqref{system} that will be given below.
\end{theorem}
\begin{proof}
Multiplying both sides of \eqref{BASIC} by the denominator $\LD B(s)$ gives
\begin{align}
&\LD B(s)\big(\gen W(r,s) - \e^{-sw}\big)
\nonumber\\&\qquad\qquad = r\LN B(s)\,\LL A(-s) \,\gen{\VV W}(r,s) +  r\LD B(s)\Big(\frac{1}{1-r} - \gen\WA(r,s)\Big) .
\label{recur5}
\end{align}
Now observe the following:
\begin{enumerate}[label=(\roman*)]
\item the left-hand side of \eqref{recur5} is analytic in $\re s >0$ and continuous in $\re s \geqslant 0$,
\item  the right-hand side of \eqref{recur5} is analytic in $\re s <0$ and continuous in $\re s \leqslant 0$,
\item for large $s$, both sides are $O(s^l)$ in their respective half-planes
%\footnote{\scriptsize \tt MM: what is $n$? Don't you mean $\ell$?}.
\end{enumerate}
At the boundary $\re s=0$, both sides are well-defined.
Determination of the unknown functions $\gen W(r,s)$ and $\gen{\WA}(r,s)$ from \eqref{recur5} and conditions (i), (ii) and (iii) is
a Wiener-Hopf boundary value problem
of a type that has been extensively studied in the queuing theory before, cf.\ the expository paper \cite{Cohen75}.
By introducing a function $G(r,s)$ that is equal to the left-hand side of \eqref{recur5} for $\re s \geqslant 0$ and to the right-hand side
of \eqref{recur5} for $\re s \leqslant 0$, we have a function that is analytic in the whole $s$-plane, and that for large $s$ is $O(s^{\ell})$.
Liouville's theorem  \cite[p.\ 85]{Titchmarsh} now states that both sides of \eqref{recur5}, in their respective half-planes, are equal to the same $\ell$-th degree polynomial in $s$. In other words,
\begin{align}
\LD B(s)\big(\gen W(r,s) - \e^{-sw}\big)  =\sum_{k=0}^\ell a_k(r) s^k
\label{recur6}
\end{align}
for $\re s\geqslant 0$ and
\begin{align}
r\LN B(s)\,\LL A(-s)\, \gen{\VV W}(r,s) +   r\LD B(s)\Big(\frac{1}{1-r} - \gen\WA(r,s)\Big) =\sum_{k=0}^\ell a_k(r) s^k
\label{recur7}
\end{align}
for $\re s \leqslant 0$. We still need to determine the $\ell+1$ unknown functions $a_0(r),\dots,a_\ell(r)$.
Taking $s=0$ in either \eqref{recur6} or \eqref{recur7} gives the expression in (\ref{recur8}) for $a_0(r)$. Next we take $s=s_j$, $j=1,\dots,\ell$. We do this in \eqref{recur7}, observing that $\re s_j<0$.  Using that $\LD B(s_j)=0$ we thus obtain
\begin{align}
r\LN B(s_j)\,\LL A(-s_j)\,\gen{\VV W}(r,s_j)  = \sum_{k=0}^\ell a_k(r) s_j^k , \quad j=1,\dots,\ell.
\label{recur9}
\end{align}
Applying \eqref{note}, this identity can be rewritten into
\begin{align}
r\LN B(s_j)\,\LL A(-s_j)\,  \int_{-\infty}^0 \gen W(r,s_jy) \,\PP(\VV\in\,{\rm d}y)  = \sum_{k=0}^\ell a_k(r) s_j^k , \quad j=1,\dots,\ell.
\label{recur10}
\end{align}
Using \eqref{recur6}, Equation \eqref{recur10} becomes, for $j=1,\dots,\ell$,
\begin{align}
r\LN B(s_j)\,\LL A(-s_j) \int_{-\infty}^0 \Big(\e^{s_j yw} + \frac{\sum_{k=0}^\ell a_k(r) (s_jy)^k}{\LD B(s_jy)}\Big)\,\;\PP(\VV\in{\rm d}y)  = \sum_{k=0}^\ell a_k(r) s_j^k.
\label{recur11}
\end{align}
We can rewrite this equation as follows: for $j=1,\dots,\ell$,
\begin{align}
& \sum_{k=0}^\ell a_k(r) s_j^k\left(
1 - r\LN B(s_j)\,\LL A(-s_j) \int_{-\infty}^0 \frac{y^k}{\prod_{m=1}^\ell (s_jy - s_m)}\;\PP(\VV\in{\rm d}y) \right)
\nonumber
\\
&\qquad\qquad  = r\,\LN B(s_j)\,\LL A(-s_j)\,\LL{\VV}(-s_jw).
\label{system}
\end{align}
One can determine the remaining unknowns $a_1(r),\dots,a_\ell(r)$ from this set of $\ell$ linear equations. Subsequently, from \eqref{recur6}, \eqref{W1} follows. Expression \eqref{W2} then follows from \eqref{recur7}.
\end{proof}

We proceed by discussing the stationary behavior of $\seq{W}$.

\begin{theorem} Suppose that the Conditions {\em (A)} and {\em (B)} hold. If $\pp{B\leqslant A}>0$ then $W_i$ converges weakly to a proper limit $W$ as $i\to\infty$, and
\begin{align}
\LL{W}(s)=\frac{\sum_{k=0}^\ell a_k s^k}{\LD B(s)},
\label{lstat}
\end{align}
where
\begin{equation}\label{eqa0}
a_0 =(-1)^\ell\prod_{i=1}^\ell s_i,
\end{equation}
and the remaining constants  $a_1,\dots,a_\ell$ can be determined from the linear system \eqref{recur15} that will be given below.
\end{theorem}
\begin{proof}
If $\pp{B\leqslant A}>0$ holds then Condition \ref{w1} is fulfilled, so  $W_n$ weakly converges to a proper limit.
We obtain the steady-state behavior \eqref{lstat} from its transient counterpart \eqref{W1} in a standard manner, viz.\ by using an Abelian theorem for power series:
\begin{align}
\LL W(s) = \lim_{r \uparrow 1} (1-r) \,\gen W(r,s)  = \frac{\sum_{k=0}^\ell a_k s^k}{\LD B(s)}, \quad \re s \geqslant 0,
\label{recur14}
\end{align}
where $a_k := \lim_{r \uparrow 1} (1-r) a_k(r)$, for $k=0,\dots,\ell$. Using \eqref{recur8} and \eqref{system}, we readily obtain the linear system
\begin{align}
\sum_{k=0}^\ell a_ks_j^k \kla{
1 - \LN B(s_j)\LL A(-s_j) \int_{-\infty}^0 \frac{y^k}{\prod_{m=1}^\ell (s_jy - s_m)}\;\PP(\VV\in{\rm d}y) }= 0
\label{recur15}
\end{align}
for $a_j$, $j=1,\dots,\ell$.
\end{proof}
\noindent
The mean of $W$ directly follows by differentiation of \eqref{lstat}:
$\LL W'(0)=(a_1D_B(0)-a_0D_B'(0))/D_B(0)^2$, with
$D_B'(0)=
1$ if $\ell=1$ and
$D_B'(0) = -\sum_{i=1}^\ell s_i$ if $\ell=2,3,\ldots$;
hence
\begin{align}\label{of}
\ee{W}=\begin{cases}
{\displaystyle\frac{1-a_1}{a_0}},&\ell=1;\\
-{\displaystyle \frac{a_1+\sum_{i=1}^\ell s_i}{a_0}},&\ell=2,3,\ldots.
\end{cases}
\end{align}
\begin{remark}
It follows from \eqref{lstat} that $W$ is a mixture of an atom at zero (with probability $a_\ell$)
and $\ell$ exponential terms. This is not surprising: as $\VV_i<0$, the only way for $W_{i+1}$ to be positive is to have $B_i > A_i - \VV_iW_i$.
Now use the fact that $B_i$ has a phase-type distribution with $\ell$ exponential phases, in combination with the memoryless property
of the exponential distribution.
\end{remark}

\begin{remark}
When $\ell =1$, one obtains (using that $a_0=-s_1$, cf.\ (\ref{eqa0}))
\[
a_1 =  \frac{1 - \LN B(s_1) \LL A(-s_1) \int_{-\infty}^0 \frac{1}{y-1} \PP(\VV\in{\rm d}y) }{1 - \LN B(s_1) \LL A(-s_1) \left(\int_{-\infty}^0 \frac{1}{y-1} \PP(\VV\in{\rm d}y) +1\right) } .
\]
For general $\ell$, we have not been able to verify formally that the set of $\ell$ linear equations (\ref{recur15}) in the $\ell$ unknowns
$a_1,\dots,a_{\ell}$ has a unique solution (as they involve the zeroes $s_j$ and the distribution of $V$ in an intricate way); similarly for the set of equations (\ref{system}) for $a_1(r),\dots,a_{\ell}(r)$. However,
since $W_i$ has a unique limiting distribution with LST $\LL W(s)$ as $i\to\infty$, there is no reason to suspect that anomalies in this set of equations will occur.
\end{remark}

\begin{example}\label{examp} Suppose that $B$ has an exponential distribution with mean $1/\mu$. Then  $\LL B(s)=\mu/(s+\mu)$, $\ell=1$ and $s_1=-\mu$. Suppose also that $\VV=-a$ a.s. with $a>0$. We then obtain
\begin{align*}
a_0(r)= \frac{r\mu}{1-r},
\quad a_1(r)= \frac{r}{1-r}\kla{1-\frac{(1+a)r\LL A(\mu)}{1 +a+ ar\LL A(\mu)}}
-\frac{(1+a)r\LL A(\mu)}{1 +a+ ar\LL A(\mu)}e^{-a\mu w}.
\end{align*}
Multiplying with $(1-r)$ and letting $r\uparrow 1$ yields the coefficients
\begin{align*}
a_0=\mu,\qquad a_1=1-\frac{(1+a)\,\LL A(\mu)}{1 +a+ a\,\LL A(\mu)},
\end{align*}
so that the LST of $W$ is given by
\begin{align}
\LL W(s) = \frac{a_0+a_1s}{\mu+s} =\PP(W>0)\frac{\mu}{\mu+s}+ \PP(W=0),
\label{recur17}
\end{align}
where the last equality follows from $\PP(W=0)=\lim_{s\to\infty}\LL{W}(s)=a_1$. We then obtain
$\ee W=\afrac{(1-a_1)}{\mu}$ in accordance with \eqref{of}.

The case where $a=1$, yielding the Lindley-type recursion $W_{i+1}  = \aplus{B_i-A_i-W_i}$, has been extensively studied in \cite{VlasiouPhD}. We obtain for the stationary process
\begin{align*}
a_0= \mu,
\quad a_1= \frac{2-\LL A(\mu)}{2+ \LL A(\mu)},
\end{align*}
which is in agreement with  \cite[Formula (4.12), p.\ 74]{VlasiouPhD}.
It is easy to see that $\PP(W=0)=a_1$ is increasing in $a$.

For $a=0$ we have $\PP(W=0) = 1 -\LL A(\mu)$. This relation is explained by observing that now $\PP(W=0) = \PP(B<A)$,
with $B \sim {\rm exp}(\mu)$.
For $a \uparrow \infty$ we have $\PP(W=0) = 1/(1+\LL A(\mu))$, which is explained by observing that
a positive $W$ is followed by a geometric($q$) number of zeroes, with $q=\PP(B<A) = 1 -\LL A(\mu)$.
\end{example}

\section{Model II: The mixed case}
\label{Case2}

In this section we consider the following variant of the model of Section~\ref{Case1}.
We again start from the recursion \eqref{recur1}, but now assume that $\VV=a$, $a>0$ with probability $p$ and $\VV<0$ with probability $1-p$. Let
\begin{align*}
\VM=(\VV\,|\,\VV<0).
\end{align*}
We keep the assumption that $B$ has a rational LST, but add the requirement that $A$ has a rational LST. Summarizing, we impose the conditions
\begin{itemize}
\item[(A*)] either $\VV<0$ or $\VV=a>0$ a.s.,
\item[(B)] $\LL B\in \QQ{s_1,\dots,s_\ell}$ with $\re s_j<0$ for $j=1,\dots,\ell$,
\item[(C)] $\LL A\in \QQ{t_1,\dots,t_m}$ with $\re t_i<0$ for $i=1,\dots,m$.
\end{itemize}

\begin{theorem}
Suppose that the Conditions {\em (A*)}, {\em (B)}, and {\em (C)} hold. Then, for $r\in(0,1)$,
\begin{enumerate}
\item\label{uno} if  $a=1$ then
\begin{align}\label{U}
\gen W(r,s)
=\frac{\LD Y(s)\,\e^{-sw}+\sum_{k=0}^{m+\ell} a_k(r) s^k}{\LD Y(s)-rp\,\LN Y(s)};
\end{align}
where
\begin{equation}\label{eqa0r}
a_0(r) = \frac{r}{1-r} (1-p) (-1)^{\ell+m}\prod_{j=1}^\ell s_j \prod_{i=1}^m t_i,
\end{equation}
while the remaining constants  $a_1(r),\dots,a_{m+\ell}(r)$ can be determined from the  linear systems \eqref{rec6} and \eqref{new1} that will be given below.
%\footnote{\scriptsize\tt MM: is this correct? Strangely, only $m$ equations were mentioned in the previous draft...}
\item if $a\not=1$ then
\begin{align}
\gen W(r,s) =
\sum_{h=0}^{\infty} \Bkla{\e^{-a^hsw} + \frac{\sum_{k=0}^{m+\ell} a_k(r) (a^hs)^k}{\LD Y(a^hs)}}(rp)^h \prod_{j=0}^{h-1}\LL Y(a^js),
\label{rec14}
\end{align}
where $
a_0(r)$ is as in \eqref{eqa0r} and the remaining constants  $a_1(r),\dots,a_{m+\ell}(r)$ can be determined from the  linear systems
\eqref{lin} and \eqref{lin2} that will be given below.
\end{enumerate}
\end{theorem}
\begin{proof}
In this situation
\begin{align}
\gen{\VV W}(r,s) = p\, \gen W(r,as) + (1-p) \int_{-\infty}^0 \gen W(r,sy)\;\PP(\VM\in{\rm d}y) .
\label{rec1}
\end{align}
Then \eqref{BASIC} becomes, after multiplication by $\LD Y(s)$,
\begin{align}
&\LD Y(s) \bkla{\gen W(r,s) - \e^{-sw}} - r p\LN Y(s)\, \gen W(r,as)\nonumber
\\&\quad=\ r(1-p)\,\LN Y(s)  \int_{-\infty}^0 \gen W(r,sy)\;\PP(\VM\in{\rm d}y)+ r\LD Y(s)\Bkla{  \frac{1}{1-r} -  \gen\WA(r,s)} .\label{rec2}
\end{align}
Now the following is true:
\begin{enumerate}[label=(\roman*)]
\item the left-hand side of \eqref{rec2} is analytic in $\re s >0$ and continuous in $\re s \geqslant 0$,
\item  the right-hand side of \eqref{rec2} is analytic in $\re s <0$ and continuous in $\re s \leqslant 0$,
\item for large $s$, both sides are $O(s^{m+\ell})$ in their respective half-planes.
\end{enumerate}
Again, both sides are well-defined at the boundary $\re s=0$, so that we have a Wiener-Hopf boundary value problem. As before, the $G(r,s)$ that is equal to the left-hand side of \eqref{rec2} for $\re s \geqslant 0$ and to the right-hand side
of \eqref{rec2} for $\re s \leqslant 0$ is analytic in the whole $s$-plane, and $G(r,s)=O(s^{m+\ell})$ for large $s$. According to Liouville's theorem both sides of \eqref{rec2}, in their respective half-plane, are equal to the same $(m+\ell)$-th degree polynomial in $s$, i.e., for $\re s \geqslant 0$
\begin{align}
&\LD Y(s) \bkla{\gen W(r,s) - \e^{-sw}} - r p\LN Y(s) \gen W(r,as)
=\sum_{k=0}^{m+\ell} a_k(r) s^k, \label{rec3}
\end{align}
and for $\re s \leqslant 0$
\begin{align}
r(1-p)\,\LN Y(s)  \int_{-\infty}^0 \gen W(r,sy)\;\PP(\VM\in{\rm d}y) + r\LD Y(s)\Bkla{  \frac{1}{1-r} -  \gen\WA(r,s)}=\sum_{k=0}^{m+\ell} a_k(r) s^k.\label{rec4}
\end{align}
Taking $s=0$ in either \eqref{rec3} or \eqref{rec4} yields, after a straightforward calculation, the expression for $a_0(r)$ in \eqref{eqa0r}.
Next we set $s=s_j$, $j=1,\dots,\ell$ in \eqref{rec4}. Since $\LD B(s_j)=0$ it follows that
\begin{align}
r(1-p)\LN Y(s_j)  \int_{-\infty}^0 \gen W(r,s_jy)\;\PP(\VM\in{\rm d}y)
 = \sum_{k=0}^{m+\ell} a_k(r) s_j^k , \quad j=1,\dots,\ell.
\label{rec6}
\end{align}
We thus have obtained $\ell$ linear equations in the remaining $m+\ell$ unknown $a_k(r)$;
however, they are expressed in the yet unknown function $\gen W(r,\cdot)$.

We turn to \eqref{rec3}, which provides a relation between $\gen W(r,s)$ and $\gen W(r,as)$.
As it turns out, we have to  distinguish between the two cases $a=1$ and $a\not=1$.

\begin{itemize}
\item[$\circ$] {\em Case i:}\
For $a=1$, after division by the denominators, Relation \eqref{rec3}  can be rewritten as
\begin{align}
\gen W(r,s)\bkla{1-rp\LL Y(s)}
=\e^{-sw}+\frac{\sum_{k=0}^{m+\ell} a_k(r) s^k}{\LD Y(s)}. \label{rec7}
\end{align}
Cohen \cite{cohen}, in his study of the K$_m$/G/1 queue, proves that the term between brackets in the left-hand side of \eqref{rec7} has $m$ zeroes $\delta_1(r),\dots,\delta_m(r)$ in the right half plane $\re s > 0$.
%(actually, in $\re s >0$ if $p|r|<1$).
The analyticity of $\gen W(r,s)$ for $\re s \geqslant 0$ now implies that the right-hand side of \eqref{rec7} must be zero for all these $m$ zeroes. This results in the $m$ linear equations
\begin{equation}
\sum_{k=0}^{m+\ell} \delta_i^k(r)a_k(r) = -\e^{-\delta_i(r) w}\LD Y(\delta_i(r)), \quad i=1,\dots,m.\label{new1}
\end{equation}
%where $\delta_1(r),\dots,\delta_m(r)$ are the $m$ solutions of $\LL Y(s)=1/(rp)$ in the right halfplane and
Formula \eqref{rec6} contains $\ell$ more equations in the $a_k(r)$. Relying on \eqref{rec7},  we can rewrite it into
\begin{align}\label{lionel}
r(1-p)\LN Y(s_j)  \int_{-\infty}^0 \frac{\e^{-s_jyw}+\sum_{k=0}^{m+\ell} a_k(r) \frac{(s_jy)^k}{\LD Y(s_jy)}}{1-rp\LL Y(s_jy)}\;\PP(\VM\in{\rm d}y)
 = \sum_{k=0}^{m+\ell} a_k(r) s_j^k,
\end{align}
for $j=1,\dots,\ell$. From this we obtain
\begin{equation}
\sum_{k=0}^{m+\ell} c_{k}(r,s_j)\,a_k(r)=\int_{-\infty}^0 \frac{\e^{-s_jyw}}{1-rp\LL Y(s_jy)}\;\PP(\VM\in{\rm d}y),\quad  j=1,2,\dots,\ell\label{new2},
\end{equation}
%\eqref{new2} follows.
where, for $j=1,\dots,\ell$,
%\footnote{\scriptsize \tt MM: it said $c_{k}(r,s)$, but this can't be true... I think you meant $c_{k}(r,s_j)$, with the appropriate definition. }
\begin{align*}
c_{k}(r,s_j)=\frac{s_j^k}{r(1-p)\,\LN Y(s_j)}-\int_{-\infty}^0 \frac{(s_jy)^k}{\LD Y(s_jy)-rp\,\LN Y(s_jy)}\;\PP(\VM\in{\rm d}y).
\end{align*}

\item[$\circ$] {\em Case ii:}\ For $a<1$, Relation \eqref{rec3} has the same structure as  \cite[Formula (2.3)]{BMR}. Proceeding in a similar way as in \cite{BMR},
we write
\begin{align}
\gen W(r,s) = K(r,s) \,\gen W(r,as) + L(r,s),
\label{rec11}
\end{align}
with
\begin{align}
K(r,s) := rp\,\LL Y(s) , \quad L(r,s) := \e^{-sw} + \frac{\sum_{k=0}^{m+\ell} a_k(r) \,s^k}{\LD Y(s)} ,
\label{rec12}
\end{align}
and iteration of \eqref{rec11} now yields
\begin{align}
\gen W(r,s) =
\sum_{h=0}^{\infty} L(r,a^hs) \prod_{j=0}^{h-1} K(r,a^js) ,
\label{rec13}
\end{align}
where convergence of the infinite sum can be proven using the d'Alembert test. Indeed, for $a<1$ the limit as $h\to \infty$ of the ratio of two successive terms is
\begin{align}
\lim_{h\to\infty}\Big\vert\frac{L(r,a^hs) }{L(r,a^{h+1}s)K(r,a^hs)}\Big\vert=\frac{1}{rp}>1,\label{dala}
\end{align}
while for $a>1$, $K(r,a^hs)\to 0$ and $\abs{L(r,a^hs)}\to a_{m+\ell}(r)$, causing divergence of the left-hand side in \eqref{dala} to infinity. Insertion of \eqref{rec12} in \eqref{rec13} gives \eqref{rec14}.
The only unknowns are $a_1(r),\dots,a_{m+\ell}(r)$.
We obtain $m$ linear equations in the unknown $a_k(r)$ by observing that substitution of $s =-t_i$, $i=1,\dots,m$, in \eqref{rec3} results in the following identity:
\begin{align}
-rp\LN Y(-t_i) \gen W(r,-at_i) = \sum_{k=0}^{m+\ell} a_k(r) (-t_i)^k, \quad i=1,\dots,m.
\label{rec15}
\end{align}
Substituting the right-hand side of \eqref{rec14}, with $s = -at_i$, into \eqref{rec15} now gives the $m$ linear equations
\begin{align}
&\sum_{k=0}^{m+\ell} (-t_i)^k a_k(r)
\kla{1+rp\,\LN Y(-t_i) \sum_{h=0}^{\infty}\frac{a^{k(h+1)}(rp)^h \prod_{j=0}^{h-1}\LL Y(-a^{j+1}t_i) }{\LD Y(-a^{h+1}t_i)} }
\nonumber\\
&\quad=-rp\LN Y(-t_i)
\sum_{h=0}^{\infty} \e^{a^{h+1}t_iw}(rp)^h \prod_{j=0}^{h-1}\LL Y(-a^{j+1}t_i)
\label{lin}
\end{align}
for $i=1,\dots,m$. The remaining $\ell$ equations are provided by substituting \eqref{rec14} into \eqref{rec6}, yielding
\begin{align}
&\sum_{k=0}^{m+l} d_{k}(r,s_j)a_k(r)
\nonumber\\&=r(1-p)\LN Y(s_j) \sum_{h=0}^{\infty}(rp)^h \int_{-\infty}^0
 \e^{-a^hs_jyw}  \prod_{i=0}^{h-1}\LL Y(a^is_jy) \;\PP(\VM\in{\rm d}y)
\label{lin2}
\end{align}
for $j=1,\dots,\ell$, where
\begin{align*}
&d_{k}(r,s_j)=s_j^k\kla{1-r(1-p)\LN Y(s_j)\sum_{h=0}^{\infty}(rp)^h \int_{-\infty}^0
\frac{ (a^hy)^k \prod_{i=0}^{h-1}\LL Y(a^is_jy)}{\LD Y(a^hs_jy)}   \;\PP(\VM\in{\rm d}y)}.
\end{align*}
\end{itemize}This finishes the proof. \end{proof}
\noindent
The steady-state LST of $W$ exists if $\pp{B\leqslant A}>0$, cf.\ Theorem~\ref{Theorem1}. It can again be obtained by applying an Abelian theorem.
For example in case~(ii), in which $a<1$, one gets
\begin{equation}
U_W(s) = \sum_{h=0}^{\infty} L(a^hs) \prod_{j=0}^{h-1}  K(a^js),
\end{equation}
with
\begin{equation}
K(s) := p\Phi_Y(s), ~~~L(s) := \frac{\sum_{k=0}^{m+l} a_k s^k}{D_Y(s)},
\end{equation}
where $a_k := {\rm lim}_{r \uparrow 1} (1-r) a_k(r)$. As this argumentation mimics the line of reasoning presented in the previous section, we omit details here.

 \section{Model III: the uniform proportional case}\label{sec:intro}
\label{Case3}

In this section we once more consider the stochastic recursion
$W_{i+1} = \aplus{\VV_{i} W_i + Y_i}$, where $Y_i=B_i-A_i$.  Again we impose the usual independence assumptions on the sequences $\seq \VV$, $\seq A$, and $\seq B$. In addition, we assume that the $A_i$ are exp($\la$) distributed. The `multiplicative adjustments' $\seq \VV$ are assumed to form a sequence of unit uniformly distributed random variables on $[0,1]$. By Thm.\ \ref{theo2}, since $\ee{\log\abs V}<0$,  a steady-state distribution of $\seq W$ always exists. We shall first study its transient distribution, and then obtain the steady-state distribution.

We start with \eqref{recurc}, i.e.,
\begin{align*}
\LL{W_{i+1}}(s) =\LL{\VV_iW_i+B_i-A_i}(s) +1-\LL{\WA_i}(s), \quad i=0,1,\dots,
\end{align*}
where as before $\WA_i=\aminus{\VV_iW_i+B_i-A_i}$. This time the distribution of $\WA_i$ is almost trivial: either $\VV_iW_i+B_i-A_i\geqslant 0$, in which case $\WA_i=0$, or $\WA_i$ has the same exponential distribution as $A_i$, due to the lack of memory property of the exponential distribution.

Using the independence between $\seq A$, $\seq B$, and $\seq \VV$ and the exponentiality of the $A_i$,
we obtain
\begin{align}\nonumber\LL{W_{i+1}}(s) &= \LL{V_iW_i}(s)\LL{B}(s)\frac{\la}{\la -s} + 1- \PP(\VV_{i}W_{i}+B_i-A_i<0) \frac{\la}{\la-s}-\PP(\VV_{i}W_{i}+B_i-A_i\geqslant  0)  \nonumber \\
&= \LL{V_iW_i}(s)\LL{B}(s)\frac{\la}{\la -s}  -  p_{i+1} \frac{s}{\la-s},\label{REC}
\end{align}
where we set $p_i:= \PP(W_{i}=0)$. Our goal is to write \eqref{REC} fully in terms of the functions $\LL{W_i}(s)$. To this end, performing the change of variable $v:=su$, we obtain
\begin{align}
\LL{V_iW_i}(s) &=  \int_0^1 \ee{\e^{-suW_i}}\,{\rm d}u=\frac{1}{s} \int_0^s \LL{W_i}(v)\,{\rm d}v.
\label{uni-eq}
\end{align}
By multiplying with $\la-s$, we thus obtain the following recursive integral equation.

\begin{lemma}
For $i\in{\mathbb N}$,
\begin{align}\label{maineq}
\LL{W_{i+1}}(s) = \frac{\la\,\LL{B}(s)}{s(\la-s)} \int_0^{s}
\LL{W_i}(v)\,{\rm d}v-\frac{s}{\la-s}p_{i+1}.
\end{align}
\end{lemma}
Since the LST $\LL{W_0}(s)=\e^{-sw}$ of $W_0$ is known, Relation (\ref{maineq}) in principle allows us to recursively determine all the transforms $\LL{W_i}(\cdot)$, $i\in{\mathbb N}$.
 Observe that, when $s=\la$, the right-hand side should become zero; using (\ref{uni-eq}) we obtain
\begin{align}\label{47a}
p_{i+1} = \frac{\LL{B}(\la)}{\la} \int_0^{\la} \LL{W_i}(v)\,{\rm d}v.
\end{align}
This formula can easily be interpreted probabilistically, using the memoryless property of the exponential distribution for $A_i$:
\begin{align*}
\PP(W_{i+1}=0) = P(A_i \geqslant B_i + V_iW_i) = \PP(A_i \geqslant B_i) P(A_i \geqslant V_iW_i) = \LL{B}(\la) \LL{V_iW_i}(\la).
\end{align*}
It is not possible to obtain explicit expressions for $\LL{W_i}$. However, as so often, one can utilize the method of generating functions to turn the recursion \eqref{maineq} into some sort of differential or integral equation.   Therefore we multiply Equation \eqref{maineq} by $r^{i+1}$ and sum over $i$ to obtain
\begin{align}
\gen W(r,s)= \frac{\la r\,\LL{B}(s)}{s(\la-s)} \int_0^{s}
\gen W(r,v)\,{\rm d}v+K(s),
\label{UW}
\end{align}
where, to simplify the notation,
\begin{align}\label{kay}
K(s)=\LL{W_0}(s)-\frac{s}{\la-s}\kla{\gen W(r,\infty)-p_0}.
\end{align}
With $I(s)=\int_0^{s}\gen W(r,v)\,{\rm d}v$ we obtain the linear first order differential equation
\begin{align}\label{machine}
I'(s) = \frac{\la r\LL{B}(s)}{s(\la-s)}I(s)+K(s).
\end{align}
As follows by standard techniques, this inhomogeneous differential equation is solved by
\begin{align}\label{ddI}
I(s)=\fexp{\la r \int_c^s\frac{\LL{B}(t)}{t(\la-t)}\,{\rm d}t }
\kla{\theta(c)+\int_c^s K(u)\,
\fexp{-\la r\int_c^u\frac{\LL{B}(t)}{t(\la-t)}\,dt } \,{\rm d}u},
\end{align}
where necessarily $\theta(c)=I(c)$ and we assume for the time being that $c\in(s,\la)$ if  $s<\la$ and $c\in(\la,s)$ if $\la<s$. Since $\LL B$ is bounded and bounded away from zero, we have as $s\downarrow \la$ (and likewise if $s\uparrow \la$ in the case where $s<\la$)
\begin{align*}
\fexp{\la r \int_c^s\frac{\LL{B}(t)}{t(\la-t)}\,{\rm d}t }\to \infty.
\end{align*}
As a consequence, to make sure $\gen W(r,\lambda)$ remains bounded,
\begin{align*}
\theta(c)=-\int_c^\la K(u)\,
\fexp{-\la r\int_c^u\frac{\LL{B}(t)}{t(\la-t)}\,{\rm d}t } \,{\rm d}u,
\end{align*}
and therefore
\begin{align}
I(s)&=\fexp{\la r \int_c^s\frac{\LL{B}(t)}{t(\la-t)}\,{\rm d}t }
\int_\la^s K(u)\,
\fexp{-\la r\int_c^u\frac{\LL{B}(t)}{t(\la-t)}\,{\rm d}t } \,{\rm d}u \nonumber\\
&=\int_\la^s K(u)\,
\fexp{\la r\int_u^s\frac{\LL{B}(t)}{t(\la-t)}\,{\rm d}t } \,{\rm d}u.\label{apple}
\end{align}
It is to be noted that the integrand in \eqref{apple} tends to $\infty$ as $s\to\la$, which follows from the finiteness of $I(\la)$. Inserting \eqref{apple} into \eqref{machine}, recalling hat $\gen W(r,s)=I'(s)$, yields the following expression for the generating function $\gen W$:
\begin{align}\label{applejuice}
\gen W(r,s)=K(s)+ \frac{\la r\LL{B}(s)}{s(\la-s)}\int_\la^s K(u)\,
\fexp{\la r\int_u^s\frac{\LL{B}(t)}{t(\la-t)}\,{\rm d}t } \,{\rm d}u.
\end{align}
 Plugging \eqref{kay} into  \eqref{applejuice}   we obtain
\begin{align*}
\gen W(r,s)=K(s)+ \frac{\la r\LL{B}(s)}{s(\la-s)}\int_\la^s \kla{\LL{W_0}(u)-\frac{u\kla{\gen W(r,\infty)-p_0}}{\la-u}}\,
\fexp{\la r\int_u^s\frac{\LL{B}(t)}{t(\la-t)}\,{\rm d}t } \,{\rm d}u,
\end{align*}
It remains to determine $\gen W(r,\infty)$. Keeping in mind that $\gen W(r,s)-K(s)\to 0$ by \eqref{kay} we obtain after some rearrangements,
\begin{align*}
\gen W(r,\infty)=p_0-\frac{\int_\la^\infty \LL{W_0}(u)
\fexp{-\la r\int_u^\infty\frac{\LL{B}(t)}{t(t-\la)}\,{\rm d}t } \,{\rm d}u}{\int_\la^\infty \frac{u}{u-\la}\,
\fexp{-\la r\int_u^\infty\frac{\LL{B}(t)}{t(t-\la)}\,{\rm d}t } \,{\rm d}u}.
\end{align*}
We summarize our findings in the following theorem.

\begin{theorem} For $r\in(0,1)$,
\begin{align}\label{bs}
\gen W(r,s)=K(s)+ \frac{\la r\LL{B}(s)}{s(\la-s)}\int_\la^s K(u)\,
\fexp{\la r\int_u^s\frac{\LL{B}(t)}{t(\la-t)}\,{\rm d}t } \,{\rm d}u.
\end{align}
where
\begin{align*}
K(s)=\LL{W_0}(s)+\frac{\frac{s}{\la-s}\int_\la^\infty \LL{W_0}(u)
\fexp{-\la r\int_u^\infty\frac{\LL{B}(t)}{t(t-\la)}\,{\rm d}t } \,{\rm d}u}{\int_\la^\infty \frac{u}{u-\la}\,
\fexp{-\la r\int_u^\infty\frac{\LL{B}(t)}{t(t-\la)}\,{\rm d}t } \,{\rm d}u}.
\end{align*}
\end{theorem}
We already noted that since the $\VV_i$ are uniformly distributed on $[0,1]$ we have $\ee{\log\abs{V}}<0$. Hence we always have $W_i \dto W$ as  $i\to\infty$ for some proper random variable $W$. Its LST is given in the following theorem.

\begin{theorem} $W_i$ converges weakly to a proper limit $W$ as $i\to\infty$, and
\begin{align}\label{LL}
\LL W(s)=\frac{p_\infty}{s-\la}\kla{s- \frac{\la \LL{B}(s)}{s}\int_\la^s \frac{u}{u-\la}\,
\fexp{\la\int_u^s\frac{\LL{B}(t)}{t(\la-t)}\,{\rm d}t } \,{\rm d}u}.
\end{align}
where
\begin{align}
\label{K}
p_\infty
= \left[ \int_0^\la \frac{1}{\la-u}\,
\fexp{-\int_0^u\kla{\frac{\la\LL{B}(t)}{t(\la-t)}-\frac{1}{t} }\,{\rm d}t} \,{\rm d}u\right]^{-1}.
\end{align}

\end{theorem}
\begin{proof}
We apply an Abelian theorem and obtain \eqref{LL} after multiplying  both sides of (\ref{bs}) by $1-r$ and letting $r$ tend to one. We thereby use the fact that
\begin{align*}
\lim_{r \uparrow 1} (1-r) K(s)=\frac{sp_\infty}{s-\la}.
\end{align*}
The relation for $p_\infty$ follows by noting that $\LL W(0)=1$, so that
\begin{align*}
\frac{1}{p_\infty}
&=\lim_{s\downarrow 0} \frac{1}{s-\la}\kla{s- \frac{\la \LL{B}(s)}{s}\int_\la^s \frac{u}{u-\la}\,
\fexp{\la\int_u^s\frac{\LL{B}(t)}{t(\la-t)}\,{\rm d}t } \,{\rm d}u}\\
&=\lim_{s\downarrow 0}  \frac{1}{s}\int_s^\la \frac{u}{\la-u}\,
\fexp{-\la\int_s^u\frac{\LL{B}(t)}{t(\la-t)}\,{\rm d}t } \,{\rm d}u
\\
&=\lim_{s\downarrow 0}  \int_s^\la \frac{1}{\la-u}\,
\fexp{-\int_s^u\kla{\frac{\la\LL{B}(t)}{t(\la-t)}-\frac{1}{t} }\,{\rm d}t} \,{\rm d}u .
\end{align*}
Now
\begin{align*}
\frac{\la\LL{B}(t)}{t(\la-t)}-\frac{1}{t}\to \frac{1}{\la}-\ee B,
\end{align*}
as $t\downarrow 0$, so that we can safely let $s\downarrow 0$ and obtain the finite and non-zero limit \eqref{K}.
\end{proof}

\begin{remark} The expected value $\ee W$ can be expressed in terms of the parameters $\la$, $\ee B$ and $p_\infty$ as follows. Letting $i\to\infty$ in \eqref{REC} yields
\begin{align}
\LL{W}(s)
&= \LL{VW}(s)\LL{B}(s)\frac{\la}{\la -s}  -   \frac{s}{\la-s}p_{\infty}.\label{lim}
\end{align}
After a rearrangement of terms this becomes
\begin{align*}
p_{\infty}=\LL{W}(s)+ \la\frac{1-\LL W(s)+\LL{VW+B}(s)-1}{s}.
\end{align*}
As $s\downarrow 0$ the right-hand side tends to
$1+\la\kla{\ee W-\ee{VW+B}}$ and since $\ee{VW+B}=\ee{V}\ee{W}+\ee{B}=\frac{1}{2}\ee{W}+\ee{B}$, we obtain
\begin{align}
\ee W=2\kla{\ee B-\frac{1-p_\infty}{\la}}.
\label{EwpW0new}
\end{align}
This yields the inequality $\pp{W=0}\geqslant 1-\la\ee B$, the value that one would get if $V_i \equiv 1$.
\end{remark}

\begin{example} Even if the $B_i$ are exp($\mu$) distributed the computation of $\gen W(r,s)$ becomes  quite involved. By means of a partial fraction expansion, after considerable calculus we obtain
\begin{align*}
\frac{u}{\la-u}\fexp{\la\int_u^s\frac{\LL{B}(t)}{t(\la-t)}\,{\rm d}t }=\frac{s}{\la-u} \left(\frac{\la -u}{\la -s}\right)^{\frac{\mu}{\la + \mu}} \left(\frac{\mu+u}{\mu+s}\right)^{\frac{\la}{\la+\mu}}.
\end{align*}
For $s<\la$ the integral can be expressed in terms of the  incomplete beta function $B(x,a,b)=\int_0^x u^{a-1}(1-u)^{b-1}\,du$:
\begin{align*}
\int_\la^s \frac{s}{u - \la} \left(\frac{\la -u}{\la -s}\right)^{\frac{\mu}{\la + \mu}} \left(\frac{\mu+u}{\mu+s}\right)^{\frac{\la}{\la+\mu}}\,{\rm d}u
&=\frac{s(\la+\mu)}{(\mu+s)^{\frac{\la}{\la+\mu}}(\la -s)^{\frac{\mu}{\la + \mu}}}
B\kla{\tfrac{\la-s}{\la+\mu},\tfrac{\mu}{\la+\mu},1+\tfrac{\la}{\la+\mu}}.
\end{align*}
\iffalse % calculation here:
\begin{align*}
\int_\la^s \frac{s}{\la-u} \left(\frac{\la -u}{\la -s}\right)^{\frac{\mu}{\la + \mu}} \left(\frac{\mu+u}{\mu+s}\right)^{\frac{\la}{\la+\mu}}\,du
&=\frac{s}{(\mu+s)^{\frac{\la}{\la+\mu}}(\la -s)^{\frac{\mu}{\la + \mu}}}\int_\la^s \left(\la -u\right)^{\frac{\mu}{\la + \mu}-1} \left(\mu+u\right)^{\frac{\la}{\la+\mu}}\,du
\\&=-\frac{s(\la+\mu)}{(\mu+s)^{\frac{\la}{\la+\mu}}(\la -s)^{\frac{\mu}{\la + \mu}}}\int_0^{\frac{\la-s}{\la+\mu}} u^{\frac{\mu}{\la + \mu}-1} \left(1-u\right)^{\frac{\la}{\la+\mu}}\,du
\\&=-\frac{s(\la+\mu)}{(\mu+s)^{\frac{\la}{\la+\mu}}(\la -s)^{\frac{\mu}{\la + \mu}}}
B\kla{\tfrac{\la-s}{\la+\mu},\tfrac{\mu}{\la+\mu},1+\tfrac{\la}{\la+\mu}}.
\end{align*}
\fi
We then obtain
\begin{align*}
\LL W(s)=p_\infty\cdot \kla{\frac{\la \mu(\la+\mu)B\kla{\tfrac{\la-s}{\la+\mu},\tfrac{\mu}{\la+\mu},1+\tfrac{\la}{\la+\mu}}}{(\mu+s)^{1+\frac{\la}{\la+\mu}}(\la -s)^{1+\frac{\mu}{\la + \mu}}}
-\frac{s}{\la-s}}.
\end{align*}
This leads to
\begin{align*}
p_\infty=\frac{\mu^{\frac{\la}{\la+\mu}}\la^{\frac{\mu}{\la + \mu}}}{(\la+\mu)B\kla{\tfrac{\la}{\la+\mu},\tfrac{\mu}{\la+\mu},1+\tfrac{\la}{\la+\mu}}},
\end{align*}
so that, at least for $s<\la$,
\begin{align}\label{H}
\LL W(s)=\frac{\mu^{\frac{\la}{\la+\mu}}\la^{\frac{\mu}{\la + \mu}}}{(\la+\mu)B\kla{\tfrac{\la}{\la+\mu},\tfrac{\mu}{\la+\mu},1+\tfrac{\la}{\la+\mu}}}\cdot \kla{\frac{\la \mu(\la+\mu)B\kla{\tfrac{\la-s}{\la+\mu},\tfrac{\mu}{\la+\mu},1+\tfrac{\la}{\la+\mu}}}{(\mu+s)^{1+\frac{\la}{\la+\mu}}(\la -s)^{1+\frac{\mu}{\la + \mu}}}
-\frac{s}{\la-s}}.
\end{align}

Unfortunately a similar expression for the $s>\la$ case is not available. Instead one obtains expressions that involve hypergeometric functions. Also it seems very hard to obtain higher moments from \eqref{H} by means of differentiation. It is possible, however, to derive a recursion formula for the moments $\omega_k:=\ee{W^k}$ (where we assume their existence for $k=1,2,\ldots,j$, say) if we start with \eqref{lim}, which in our example becomes
\begin{align}\label{lex}
(\la-s)\LL{W}(s) = \frac{\la\mu}{s(\mu+s)} \int_0^{s}
\LL{W}(v)\,{\rm d}v-sp_{\infty}.
\end{align}
For $s<0$ the expansion
$\LL{W}(-s)=\sum_{k=0}^j \frac{\omega_k}{k!}s^k+o(s^j)$ holds. Inserting this into \eqref{lex} yields
\begin{align*}
(\mu-s)(\la+s)\sum_{k=0}^j \frac{\omega_k}{k!}s^k+o(s^j) = \la\mu
\sum_{k=0}^j \frac{\omega_k}{(k+1)!}s^{k}-s(\mu+s)p_{\infty}+o(s^j), ~~~s \downarrow 0.
\end{align*}
Equating the coefficients on both sides leads to
\begin{align*}
&\la\mu\frac{\omega_k}{k!}+ (\mu-\la)\frac{\omega_{k-1}}{(k-1)!}-\frac{\omega_{k-2}}{(k-2)!}\ind{k\geqslant 2}
\\&=\kla{\la\mu\frac{\omega_1}{2}-\mu p_\infty}\ind{k=1}+\kla{\la\mu
\frac{\omega_2}{6}-p_\infty} \ind{k=2}+ \frac{\la\mu
\omega_k}{(k+1)!}\ind{k\geqslant 3}
,\quad k\in\{1,\ldots,j\}.
\end{align*}
Then, in accordance with the general result  \eqref{EwpW0new},
\begin{align*}
\omega_1&= 2\kla{\frac{1}{\mu}-\frac{1-p_\infty}{\la}}.
\intertext{Moreover,}
\omega_2&=3\frac{1-p_\infty-(\mu-\la)\omega_1}{\mu\la},\\
\intertext{and}\omega_k
&=\frac{(k^2-1)\omega_{k-2}-(k+1)(\mu-\la)\omega_{k-1}}{\la\mu},\quad k\in\{3,\ldots,j\}.
\end{align*}
\end{example}

\section{Discussion and concluding remarks}
\label{concl}

This paper has analyzed three reflected autoregressive processes specified by the stochastic recursion $W_{i+1} = [V_iW_i + B_i - A_i]^+$.
While the classical case of $V\equiv 1$ has been widely studied in the queueing literature,
our more general setting allows explicit analysis only in special cases.
The three special cases we have considered are: (i)~$V$ attains negative values only
and $B$ has a rational LST, (ii) $V$ equals a positive value $a$ with certain probability
$p\in (0,1)$ and is negative otherwise, and both $A$ and $B$ have a rational LST,
(iii) $V$ is uniformly distributed on $[0,1]$, and $A$ is exponentially distributed.
In all three cases we present transient and stationary results,
where the transient results are in terms of the transform at a geometrically distributed epoch.

Cases which might allow explicit analysis are, for example:
\begin{enumerate}
\item\label{z1} A generalization of Model III to the case in which $V=U^{1/\alpha}$, where $U$ has a uniform distribution on $[0,1]$ and $\al>0$. In this case \eqref{maineq} becomes
\begin{align*}
s^{\al-1}\LL{W}(s) = \frac{\al\la\,\LL{B}(s)}{s(\la-s)} \int_0^{s}
v^{\alpha-1}\LL{W}(v)\,{\rm d}v-\frac{s^\al}{(\la-s)} p_{i+1}.
\end{align*}
Letting $\gen W\alf(r,s)=s^{\al-1}\gen W(r,s)$ this yields
\begin{align}\label{UWA}
\gen{W}\alf(r,s)= \frac{\la r\al\,\LL{B}(s)}{s(\la-s)} \int_0^{s}
\gen{W}\alf(r,v)\,{\rm d}v+K\alf (s),
\end{align}
where
\[K\alf(s)=s^{\al-1}\LL{W_0}(s)-\frac{s^\al}{\la-s}(\gen W\alf(r,\infty)-p_0).\] Equation \eqref{UWA} is of the exact same type as \eqref{UW}, only with $r$ replaced by $r\al$.
This allows one to derive $\gen W\alf(r,s)$ in the same way as before.
\item
A combination of Models I and III,
allowing $V$ to be either negative or having a distribution as in \ref{z1}.
%\item {\color{blue}A generalization of Model II,
%in which $V_iW_i = aW_i$ in the case of positive $V_i$ is replaced by $\sum_{k=1}^{W_i} U_k$,
%where the random variables $U_k$ are i.i.d.\ Bernoulli random variables with mean $a \in [0,1]$; cf.\  the setup analyzed in \cite{WOM}.}
\item
One might consider more general recursions, such as
the high-order Lindley equations analyzed in, e.g.,
\cite{Biggins_98, Kar_Kel_Suh_94, Olv_Rui_19}.
\end{enumerate}

Another possible line of research concerns scaling limits and asymptotics.
In particular, tail asymptotics
seem to be within reach; in heavy-tailed cases these may be identified relying on a Tauberian approach.
One also anticipates that, under particular scalings, an explicit analysis is possible.
Specifically, one would expect that a diffusion analysis similar to the one
presented in \cite{BMR} can be performed.

\bibliographystyle{abbrvnat}
\bibliography{lit}
\end{document}